\documentclass[12pt]{amsart}
\usepackage{latexsym,fancyhdr,amssymb,color,amsmath,amsthm,graphicx,listings,comment}
\usepackage[section]{placeins}
\newtheorem{lemma}{Lemma} 

\let\paragraph\subsection

\title{Remarks about the Moebius-Kantor graph}
\author{Oliver Knill}
\date{May 28, 2026}
\address{Department of Mathematics \\ Harvard University \\ Cambridge, MA, 02138 }


\begin{document}
\maketitle

\begin{abstract}
The Moebius-Kantor graph MK=G(8,3) is a Cayley
graph of three non-Abelian groups, the Pauli group P(1), the semi-dihedral group SD(16), as well as the 
dihedral group D(16) of order 16. In topological graph theory,
it illustrates the Heawood number 7 of the torus and 
leads to the Tucker group Aut(MK), the unique group of genus 2. We compute the Lefschetz numbers to 
illustrate the Brouwer-Lefschetz fixed point theorem.
MK is also the dual of the 2-skeleton complex of the 3-sphere G. 
The graph represents one of  flat Clifford tori of a 
Hopf fibration in the 3-sphere G=K(2,2,2,2) reflecting that Coxeter \cite{Coxeter1950} 
saw that MK is a subgraph of the 
tesseract $G^*$. The Graph MK also has the Hopf-Rynov property: any two vertices in MK
in distance smaller than its diameter are connected by a unique shortest path if we 
take the geodesic distance given in \cite{geodesics1}, closed geodesics in $MK$ have 
length $8$. 
It carries a metric $d$ so that $(MK,d)$ has only one algebraic group structure
$(P(1),*)$ preserving the metric.  It makes the Pauli group natural 
\cite{GraphsGroupsGeometry}, similarly as the Moebius ladder $M(16)$ makes the 
dihedral group $D(16)$ natural, forcing the algebraic structure from the metric structure.
\end{abstract}

\section{Introduction}

\paragraph{}
The {\bf Moebius-Kantor graph} MK is important in {\bf topological graph theory} \cite{TuckerGross}.
It is a graph regular toroidal graph that is the Cayley graph of three different groups
P(1), D(16) and SD(16) of order 16. \footnote{There are 14 groups of order 16, five of which are Abelian}.
MK is the dual graph of a triangulation of the 2-torus. 
Similarly as the tesseract (8-cell) can be seen as a 3-sphere (16 cell) as the polytope dual 
of the cross polytop $K(2,2,2,2)$ (complete bipartite graph), the graph $MK$ can be seen as a torus that can be
embedded in the 8-cell as a discrete Clifford torus. It is naturally flat with respect to a
Higushi curvature $K(v) = 1/d_1 + 1/d_2 + 1/d_3-1/2$ (see \cite{FormCurvatures} in arbitrary 
dimensions), if the three polygons intersecting in $v$ have degrees $d_j$. 
The faces for MK all have degree $d_j=6$. The graph Gauss-Bonnet curvature $K=1-d(x)/2$ 
is constant $-1/2$ adding up to the Euler characteristic $|V|-|E|$. 


\begin{figure}[!htpb]
\scalebox{0.3}{\includegraphics{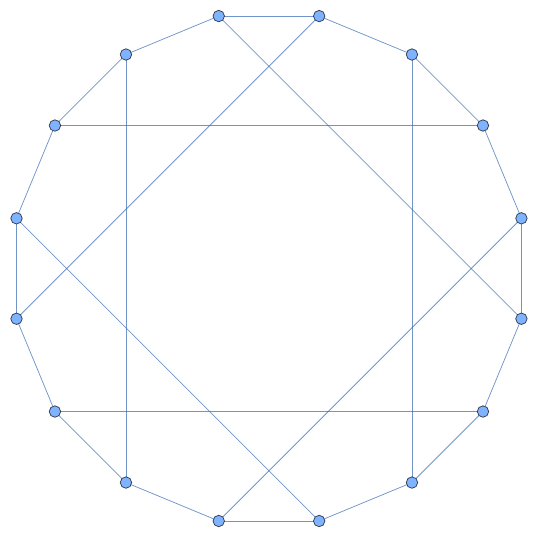}}
\scalebox{0.3}{\includegraphics{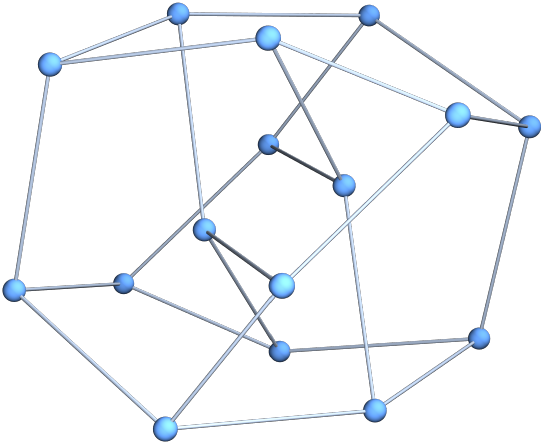}}
\scalebox{0.3}{\includegraphics{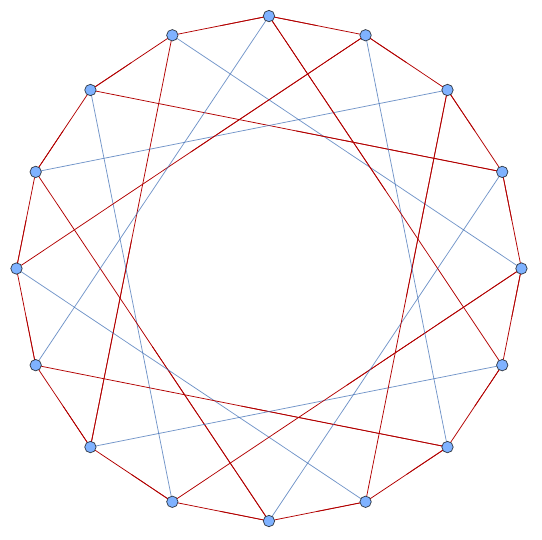}}
\scalebox{0.3}{\includegraphics{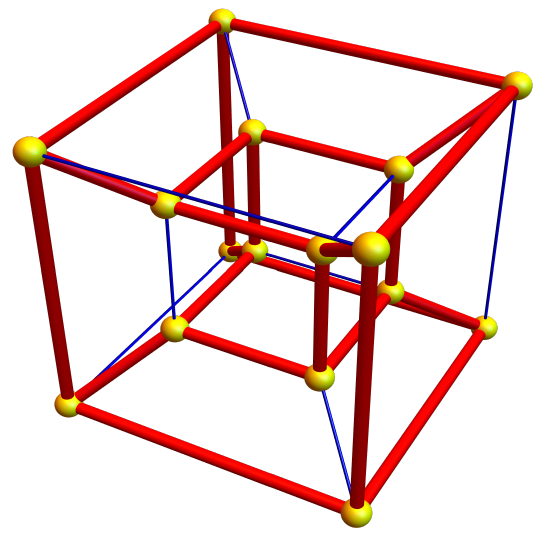}}
\label{Moebius Kantor} 
\caption{
The Moebius-Kantor graph MK is first shown with its vertices on a circle in $\mathbb{R}^2$,
then drawn in three dimensions. It is a subgraph of the circulant graph $C_{16}(1,5)$
with connection set $\pm 1,\pm 5$.
Coxeter noticed that it can be seen subgraph of the tesseract \cite{Coxeter1950}.
}
\end{figure}

\paragraph{}
H.S.M. Coxeter \cite{Coxeter1950} depicted MK as a subgraph of the {\bf tesseract} $G^*$. 
MK is obtained by deleting $8$ non-intersecting edges in $G^*$ as in Figure 1. 
The tesseract is the dual $G^*$ of the {\bf 16-cell} = {\bf cross-polytope} $G=K_{2,2,2,2}$,
the complete multi-partite graph.
The hypercube $G^*$ can be seen as the cube in 
$4$-dimensional Euclidean space $\mathbb{R}^4$ or as the Cartesian product of $K_2$.
An unfolded net of the tesseract is featured in Dal\'i's surrealist oil 
picture "Crucifixion (Corpus Hypercubus)" from 1954. \footnote{The painting is in the
Metropolitan Museum of Art in New York.}

\paragraph{}
The hypercube $G^*$ is remarkable as it has a Hamiltonian decomposition
\cite{BermondFavaronMaheo}: 
it can be decomposed as a union of two disjoint
Hamiltonian cycles, simple closed paths that cover all vertices.
Finding Hamiltonian cycles has been a {\bf mathematical game} 
{\bf Icosian game}, invented by William Rowan Hamilton on the dodecahedron. Any 
graph with this complementary Hamiltonian decomposition property must be 4-regular
and so is Eulerian by the {\bf Euler-Hierholzer theorem}. 
In such a case the two Hamiltonian cycles on the hypercube can be 
``added" (in the sense of the fundamental group) to get an {\bf Eulerian
cycle} in $G^*$, a closed path visiting each edge exactly once.
The subgraph $MK$ is still Hamiltonian but no longer Eulerian, as it is cubic. 

\begin{figure}[!htpb]
\scalebox{0.5}{\includegraphics{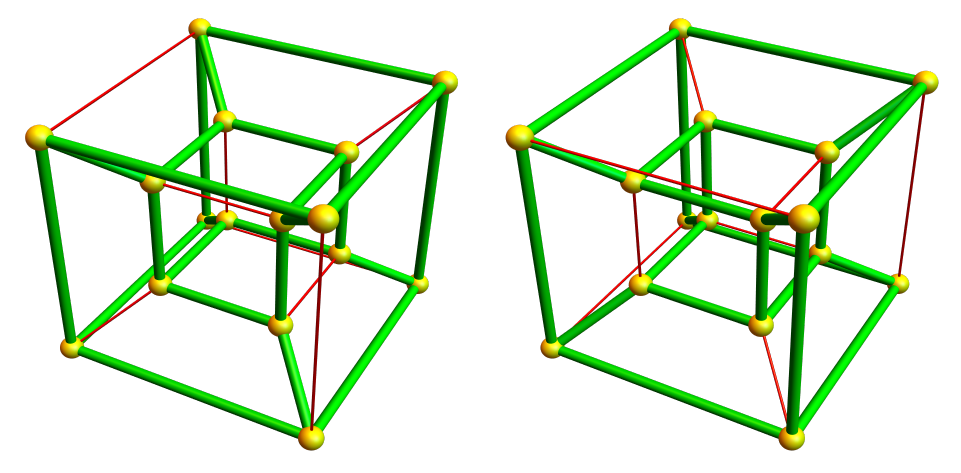}}
\scalebox{0.4}{\includegraphics{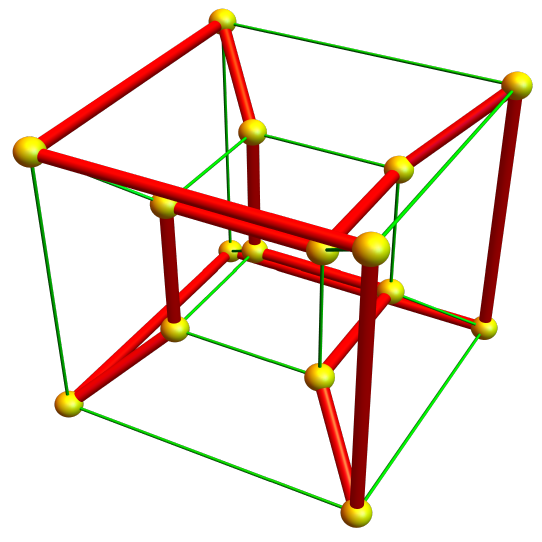}}
\caption{
Displaying 2 isometric rotated tori in the tesseract can be seen as a discrete analog
of what the Hopf fibration displays. The union of the complements
of the MK sub-tori in the tesseract $G^*$ give a Hamiltonian path in $G^*$.
Also the edge complement of this path is a Hamiltonian path. The tesseract
has a Hamiltonian decomposition of its edges into two Hamiltonian cycles.
In such a case the addition produces an {\bf Eulerian cycle}, a path covering
every edge exactly once.
}
\end{figure}

\paragraph{}
The polytopes $G,G^*$ are two of the six regular {\bf Platonic polytopes}
meaning regular convex $4$-polytopes. Besides the 16-cell $G$
and the 8-cell $G^*$, there is the 4-cell (the 3-skeleton of the hyper-tetrahedron $K_5$),
the self-dual $24$ cell $U$, (the group of units in the Hurwitz 
quaternions \cite{Hurwitz1919}), and then there is the 120-cell 
(the hyper-dodecahedron) and the 600-cell (the hyper-icosahedron). 
Ludwig Schlaefli first classified all Platonic solids in four dimensions in the 1850s.
Among these six polytopes 
(the 4,8,16,24,120,600 cells), only the 8 cell and 600 cell are discrete 
3-manifolds in the sense that they are generated by their 1-skeletons. 
\footnote{Or better, in the sense that all unit spheres are $2$-spheres.}

\begin{figure}[!htpb]
\scalebox{0.06}{\includegraphics{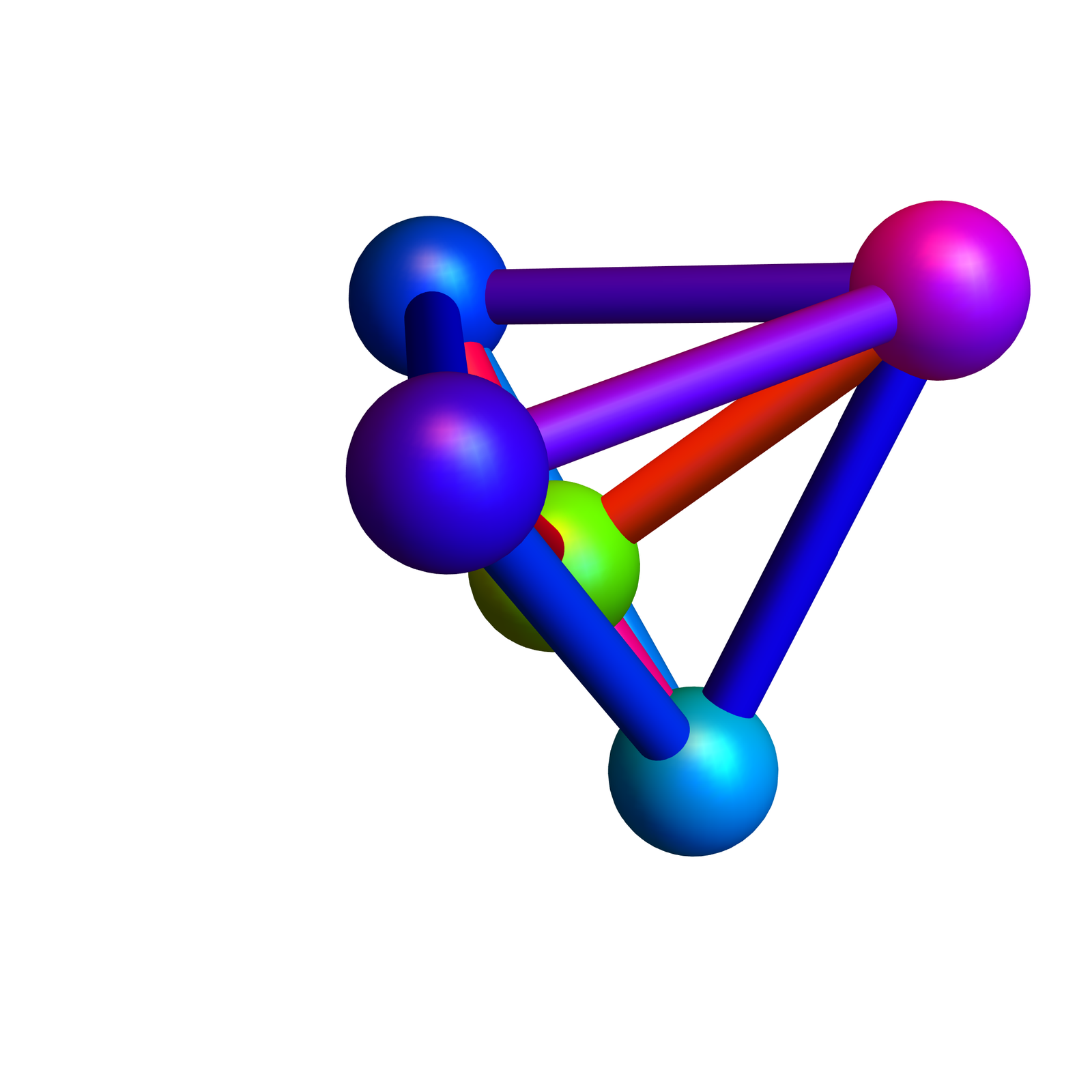}}
\scalebox{0.06}{\includegraphics{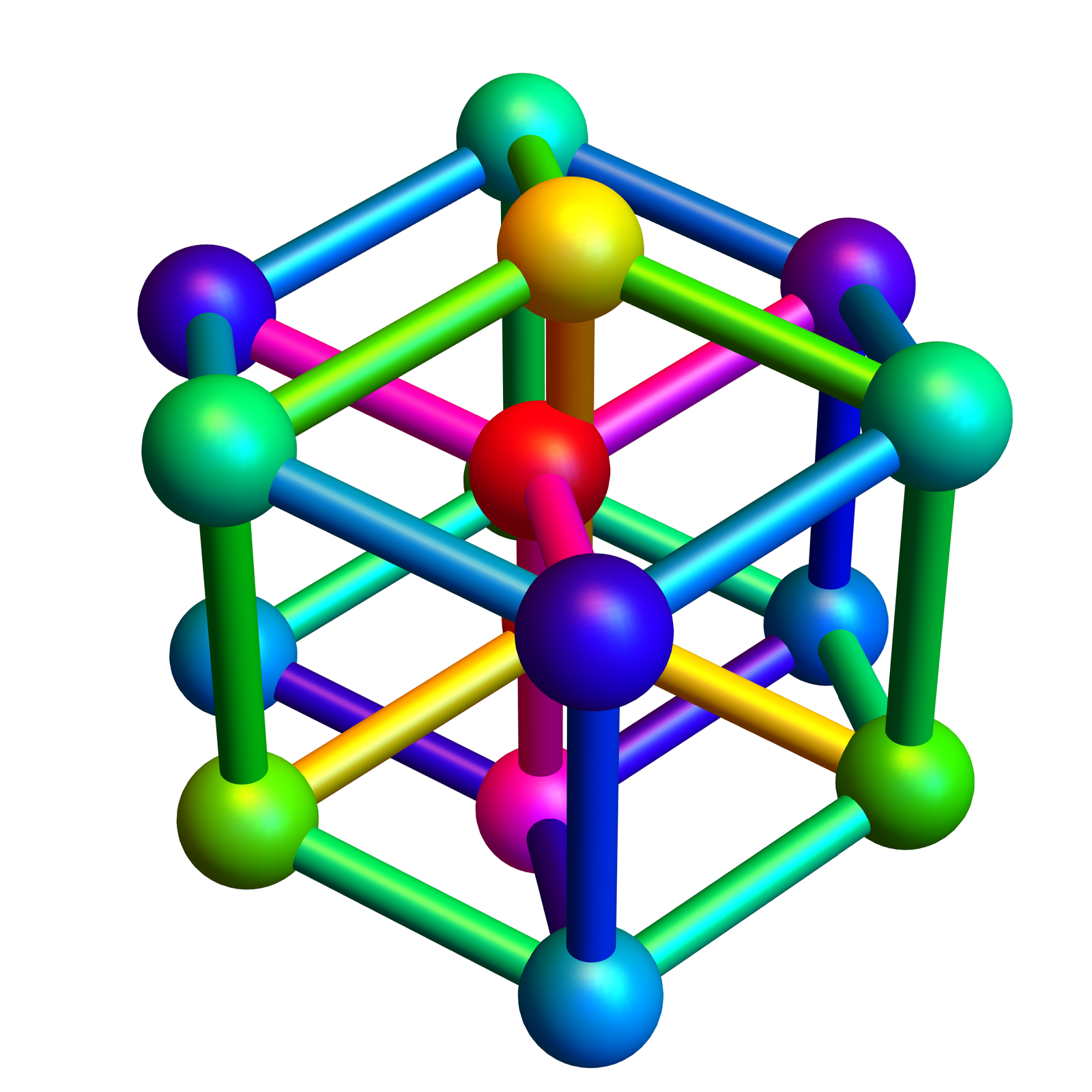}}
\scalebox{0.06}{\includegraphics{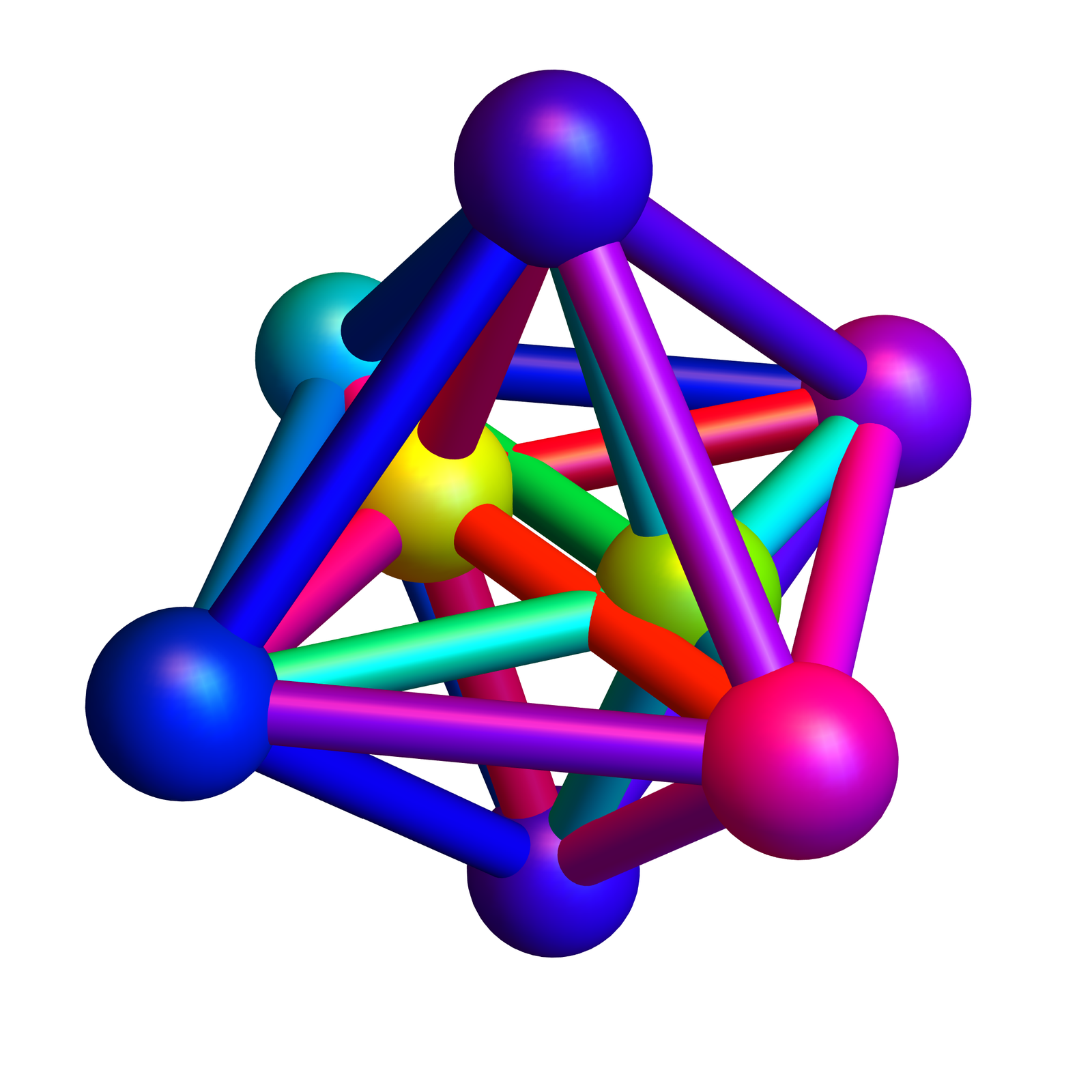}}
\scalebox{0.06}{\includegraphics{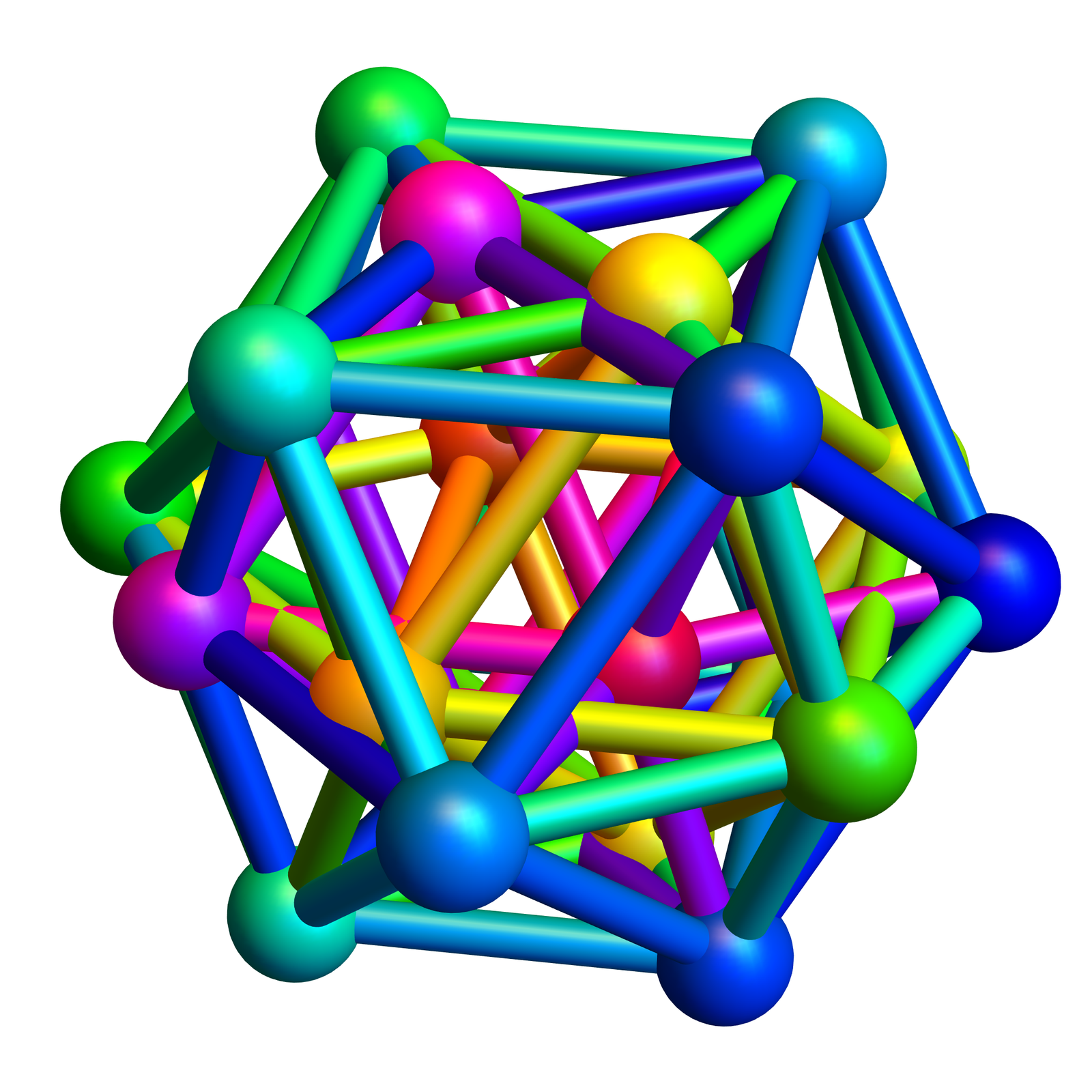}}
\scalebox{0.06}{\includegraphics{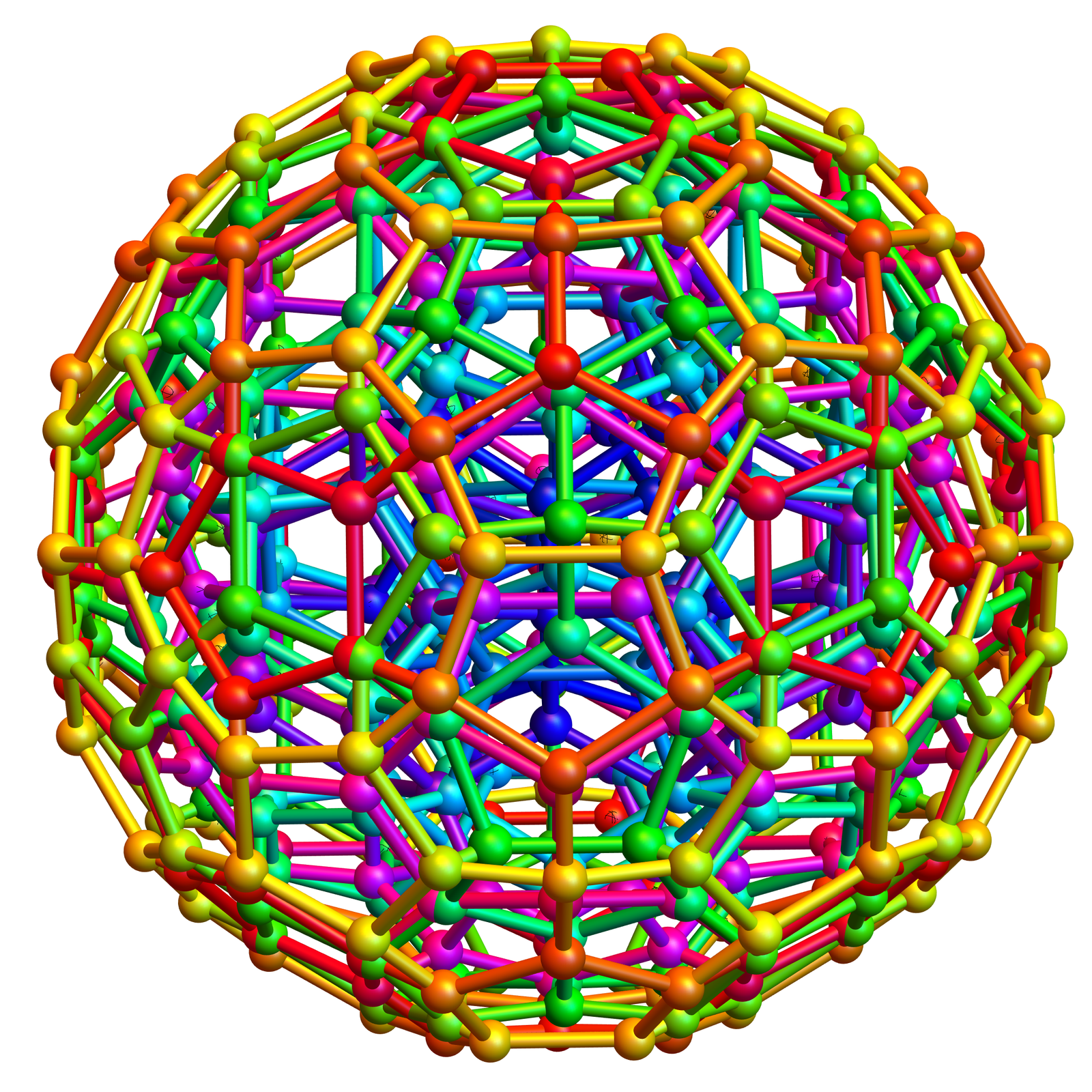}} 
\scalebox{0.06}{\includegraphics{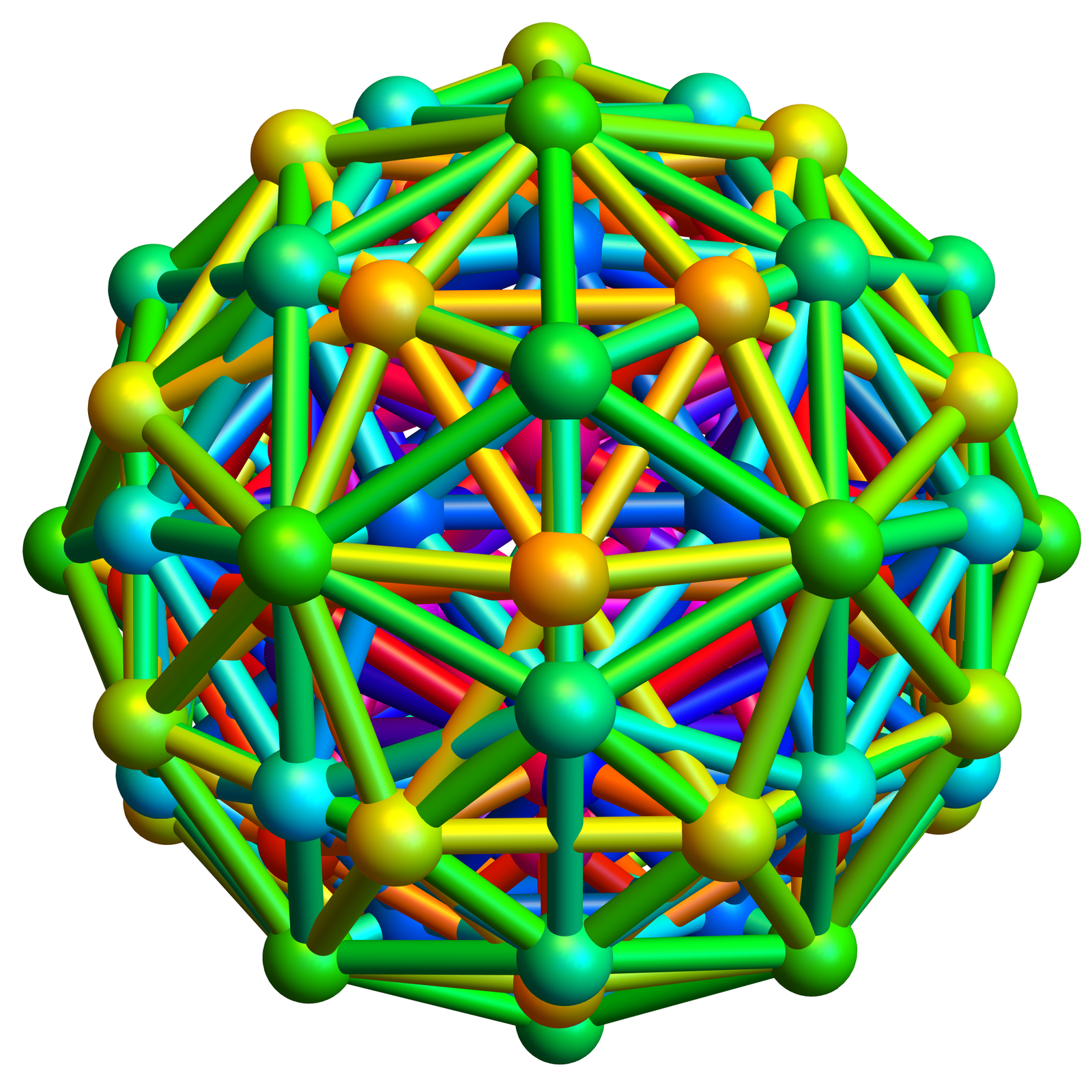}}
\label{popculture}
\caption{
The six Platonic solids in 4D. The second is the tesseract (the 8 cell), 
the third is the cross polytop (the 16 cell). 
}
\end{figure}

\paragraph{}
The graph MK is remarkable because it is the Cayley graph 
of three different non-Abelian finite groups, 
the group $SD(16) = C(8) \rtimes C(2)$, the {\bf dihedral
group} $D(16)$ and the Pauli group $P(1) = (C(4) \times C(2))  \rtimes C(2)$, 
where $\rtimes$ is a semi-direct product. 
The two groups are both non-trivial fiber bundles 
\footnote{The fiber bundle terminology is usually not used but it is intuitive. It means
semidirect product with Abelian normal subgroup.}
over $C(2)$ but the 
Abelian fibers are different groups of order $8$: every element in $C(8)$ is cyclic with
an element of order $8$, while every element in $C(4) \times C(2)$ is cyclic of order $4$. 
\footnote{This is not unusual: the smallest case is the bipartite $K(4,4)$ which is the 
  Cayley graph of both the dihedral $D(4)$ 
  and the quaternion group $Q(8)$. And $K(8,8)$ even has $D(16),SD(16),Q(16),D(8) \times C(2)$
  as Cayley graphs. And the 3-sphere $G=K(2,2,2,2)$ is the Cayley graph of both $Q(8)$ and $D(8)$.
  And $K(8)$ with two disjoint 4 cycles removed has $D(8)$ and $Q(8)$ as Cayley graphs. 
  $D(16)$ and $D(8) \times D(2)$ both can have the same Cayley graph. } 

\paragraph{}
The extension structure of the three groups is similar to the one of the {\bf Rubik's Cube}
which is known to have the structure $(C_2^{11} \times C_3^7) \rtimes S_8$ and so is also
a non-trivial Abelian fiber bundle \cite{Singmaster1981}.
The {\bf Pauli group} $P(1)$ is generated by the {\bf Pauli matrices}
$$ X=\left[\begin{array}{cc} 0 & 1 \\ 1 & 0 \end{array} \right],
   Y=\left[\begin{array}{cc} 0 & -i \\ i & 0 \end{array} \right], 
   Z=\left[\begin{array}{cc} 1 & 0 \\ 0 & -1 \end{array} \right] \; , $$
for which $iX,iY,iZ$ is a basis of $su(2)$, the Lie algebra of $SU(2)$. The latter Lie group
is topologically the 3-sphere, the only Euclidean sphere admitting a non-Abelian
group structure. 

\paragraph{}
The Pauli group $P(1)$ contains the {\bf quaternion group} $Q(8)$, the group generated by the 
unit quaternions $i,j,k$ as a subgroup. An interesting question we pursued was to 
see whether the restricted floppy group (the $3 \times 3 \times 1$ Rubik) with 192 
elements (in a reduced form) can have the Tucker
group as a subgroup. But floppy does not have any element of order $8$ so that this is not
possible. We could not yet answer the question whether the Tucker group could be realized
as a subgroup of the Rubik's Cube. 
\footnote{For the pocket Rubik's Cube, GAP can determine the structure of all subgroups
of order 96. It does not seem to contain the Tucker group. It would be interesting to see
whether in one of the known Rubik-type puzzles, the Tucker group could be realized.}

\begin{figure}[!htpb]
\scalebox{0.3}{\includegraphics{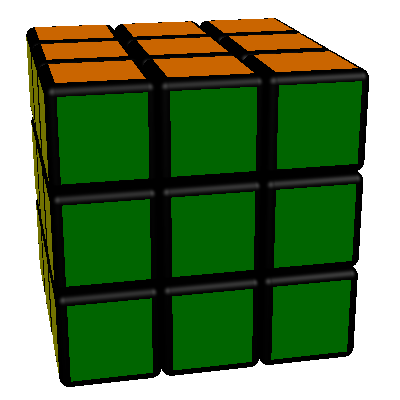}}
\scalebox{0.3}{\includegraphics{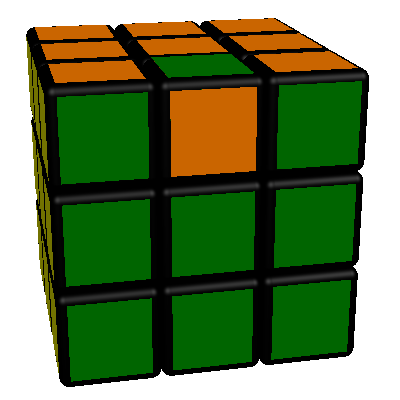}}
\label{Rubik}
\caption{
The Rubik's Cube and an impossible configuration in the form 
of a Meson (a quark-anti-quark state). 
}
\end{figure}

\paragraph{}
The graph $MK$ is {\bf toroidal}. It can be embedded in the torus 
$\mathbb{T}^2$. This embedding subdivides the torus into eight hexagonal
regions. The dual is a $6$-regular triangulation of the torus; it is
however very small because the dual graph has diameter $2$. If we identify two
of the hexagonal regions (like $5$ and $7$ in the figure), the dual is $K_7$ and 
illustrates the Heawood sharp 7-coloring bound of the torus. The graph coloring problem 
for discrete 2-manifolds \footnote{Meaning that every unit sphere is a circular graph}
is a different problem. The embedding of $MK$ in the torus is not locally 
planar in the sense of \cite{AlbertsonStromquist}. Albertson and 
Stromquist conjectured that 5 colors suffice for such 2-manifolds. 

\begin{figure}[!htpb]
\scalebox{0.5}{\includegraphics{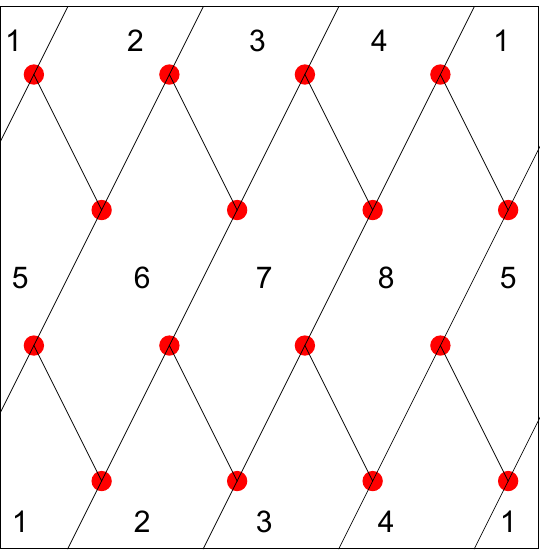}}
\scalebox{0.5}{\includegraphics{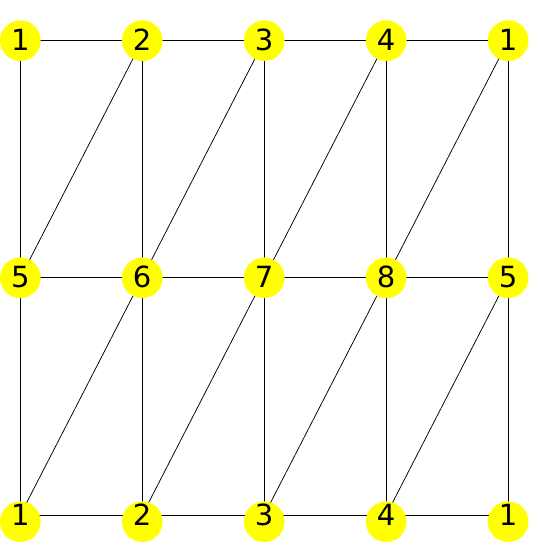}}
\caption{
The graph $MK$ on the torus. One can see the hexagonal faces= "countries". We can 
color country 5 and 7 with the same color and have a coloring of the 2-torus. 
The dual of MK is a triangulation of the torus. (But is not a 2-manifold as there is
no embedded wheel graph, wheel graphs have boundary points identified).
}
\label{heawood}
\end{figure}

\paragraph{}
A major reasons why the graph $MK$ has is important because its automorphism group 
$Aut(MK)$ is the unique group of genus $2$. 
This group has order 96 and called the {\bf Tucker group} \cite{Tucker1984}.
It can be seen as $GL(2,3) \rtimes C_2 = (Q_8 \rtimes S_3)  \rtimes C_2$.
Its {\bf White genus} $\gamma(G) = {\rm min}_S {\bf Genus}(Cayley(G,S))$ \cite{White1972}
is 2. Tucker uses the {\bf Euler formula} $V-E+F=-2$ and the regularity of 
Cayley graphs to force algebraic presentations.
\footnote{Key is the Proulx's classification from 1978 of toroidal groups 
   to discard Cayley graph that are embeddable in the torus. }
The fact that a single topological feature characterizes a finite group is remarkable. 
The Cayley graph on the genus 2 surface of this group has been built by artists. It is a 
graph that can not be embedded on the torus but can be embedded on the genus 2 torus. 

\begin{figure}[!htpb]
\scalebox{0.9}{\includegraphics{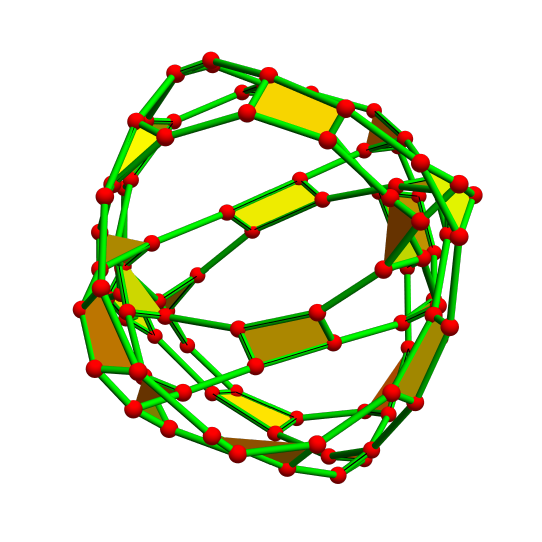}}
\label{Tucker graph}
\caption{
The Cayley graph of the Tucker group that has genus 2. 
}
\end{figure}

\section{Geodesics}

\paragraph{}
If $v,w$ are two different points on a graph, a path of minimal length that 
connects $v$ with $w$ is a {\bf graph geodesic}. There are many geodesics in general 
already for distance $2$. If $G$ is a $q$-manifold, 
and $G^*$ is the dual graph, a triangle free $(q+1)$-regular graph, we have
defined a geodesic flow on the {\bf frame bundle} $F$ of $G^*$, the set of
oriented maximal simplices. It is a permutation on $F$. It is given by 
$(x_0,x_1,x_2,\dots, x_q) \to (x_1,x_2,x_3,\dots x_0')$, where $\{x_0,x_0'\}$ is the $0$-sphere
obtained by intersecting the unit spheres $S(x_1) \cap  \cdots \cap S(x_{q})$.

\paragraph{}
As there are exactly $(q+1)!$ geodesics emanating every maximal simplex in $G$,
not every two pairs can be connected with a geodesic. When is there is a unique
geodesic between two points in $G^*$?

\paragraph{}
There are different ways to define a 2 dimensional geodesic sheet in a q-manifold $G$
in order to define sectional curvature. Sectional curvature determines the Riemann 
tensor in the continuum so that there is motivation to have a good notion of sectional
curvature also in finite manifold structures. Here are three approaches we have followed
so far. The last one involves geodesics explicitly.

\paragraph{}
{\bf A) Mickey-Mouse sectional curvature}.
The simplest is to take a wheel graph in $G$ and define sectional the curvature as 
$1-1/{\rm deg}(x)$. This extremely naive notion are many draw-backs. One reason is
that there many ways to draw a wheel graph. There are therefore
many ways to construct ``geodesic manifolds" by patching together such wheel graphs.
In any case, the assumption to have all wheel graphs to have positive curvature is very 
narrow.  Positive curvature q-manifolds all necessarily are $q$-spheres. Instead of the degree
of a wheel graph, one can also count the set of all maximal simplices hinging on a
$q-2$ simplex and so get a notion of curvature for each edge.

\paragraph{}
{\bf B) Regge construction using bones}. This is motivated by \cite{Regge,Misner}: 
A $2$ simplex $t$ in a q-manifold has 3 edges. Each $e$  
in them defines a $(q-2)$ simplex $x \setminus e$, whose dual is a 1-sphere. 
\footnote{The dual of a simplex $x$ is the intersection of all unit spheres of vertices in $x$.
If the dimension of $x$ is $k$, then the dimension of the dual is $q-k-1$.}
These three 1-spheres produce three petals of length $d_i$ and so a curvature 
$\sum_{i=1}^3 1/d_i -1/2$. When continuing the construction, one gets geodesic sheets. 
The main draw back is that these sheets are not invariant under the geodesic flow.
In the case of the graph $G=K(2,2,2,2)$, the geodesic sheets are now all octahedra. 

\begin{figure}[!htpb]
\scalebox{0.5}{\includegraphics{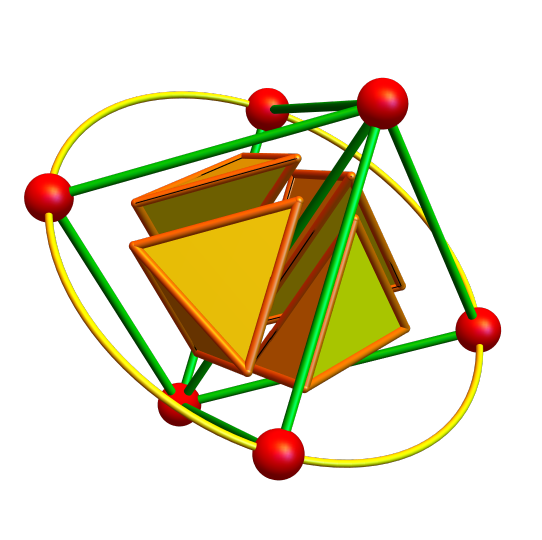}}
\scalebox{0.5}{\includegraphics{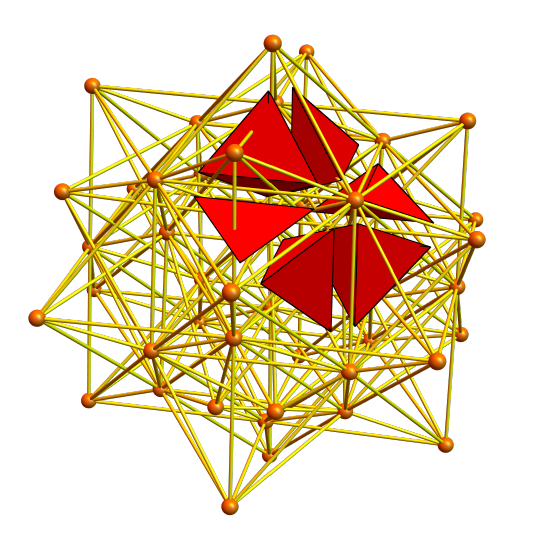}}
\label{Regge} 
\caption{
To the left we see the Regge construction in the 3-manifold $K(2,2,2,2)$. The edge degree is $4$. 
The 3-manifold seen to the right is a refinement of $K(2,2,2,2)$. Every edge has a circular
arrangement of maximal simplices hinging on it. 
}
\end{figure}

\paragraph{}
{\bf C) Geodesic sheets}. The third possibility is to use a triangle to define
$6$ geodesics and continue building a web of geodesics on those paths.
This produces a 2-manifold that is invariant under the geodesic flow. 
This is how we got to $MK$, as $MK$ appeared as such a sheet. 
And it produced a difficulty: even in highly "positive curvature manifolds" 
like the 3-sphere $K(2,2,2,2)$, the geodesic
sheets constructed as such can be flat. 

\begin{figure}[!htpb]
\scalebox{0.5}{\includegraphics{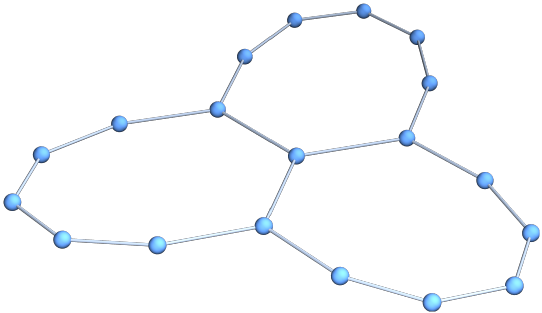}}
\label{Clover}
\caption{
This graph has been constructed by using geodesic sheets.
We see part of the $KM$ graph. 
}
\end{figure}

\paragraph{}
Let us call a sub graph $\Gamma$ in a dual $G^*$ of a discrete manifold $G$ 
{\bf Hopf-Rynov} if any two vertices $v,w$ in $\Gamma$ of distance smaller 
than the diameter can be connected by a unique geodesics, where geodesics is
understood in the sense of \cite{geodesics1}. 
This is different than the graph distance. 
There are closed loops in $MK$ of length $6$ but they are not made up of 
geodesics in the above sense. The graph $MK$ is Hopf-Rynov. 
Its diameter is $4$. The smallest closed geodesics have
length $8$. 

\section{Fixed points}

\paragraph{}
{\bf Klein's Erlangen program} promoted the point of view that the group of symmetries
determines the geometry. For graphs, the group of symmetries is in general
trivial, as there are in general no non-trivial automorphisms. The graph $MK$, however 
is an example of a graph with many automorphism. The automorphism group ${\rm Aut}(MK)$,
called the Tucker group, has 96 elements. 
The {\bf Lefschetz number} of an automorphism $T$ of a simplicial complex $G$ is defined as a
super-trace of the Koopman operator $\mathcal{U}_T f = f(T)$ on $l^2(G) \simeq \mathbb{R}^n$, where
$n=|G|$ is the number of simplices and $f(T)(x)=f(Tx)$. \footnote{If $T$ is not invertible, one
needs to take $f(T^{-1})$.} If $\mathcal{U}_T$ is restricted to a basis in ${\rm ker}(H_k)$, 
one gets a matrix $U_k$. The {\bf Lefschetz number} is defined as
$$ \chi_T(G) = {\rm str}(U) = \sum_{j=0}^q (-1)^j {\rm tr}(U_j)  \; . $$
One can apply this to any graph by looking at the simplicial complex of the graph 
in the form of the vertex sets of all complete subgraphs (which is a finite abstract
simplicial complex). 

\paragraph{}
For a one-dimensional complex like $MK$, one has $\chi_T(MK)={\rm tr}(U_0)-{\rm tr}(U_1)$. 
If $T$ is the identity, the restriction matrices $U_k$ are the identity matrices and the super trace 
simplifies to the Euler characteristic. The {\bf index} of a fixed simplex $x$ is defined as 
$i_T(x)=(-1)^{{\rm dim}(x)} {\rm sign}(T|x)$, where ${\rm sign}(T|x)$ is the sign of the
permutation that $T$ induces on the simplex $x$. We proved in \cite{brouwergraph}
the general {\bf Lefschetz fixed point theorem}
$\sum_{x, T(x)=x} i_T(x) = \chi_T(G)$ for any
{\bf simplicial map} $T$ on any simplicial complex $G$. 
\footnote{The simplicial map does not have to be invertible. In general just restrict $T$
to the attractor $\bigcap_{n \geq 0} T^n(G)$, the intersection of all images of $T$.}
In dimension $q=1$, this had been proven by Nowakowski and Rival \cite{NowakowskiRival}. 

\paragraph{}
The graph $MK$ admits $96$ automorphisms. There are four different types: \\
a) The identity is the only automorphism with Lefschetz number $-8$. 
b) There are 14 automorphisms with Lefschetz number $4$. They all have order $2$ or $3$.
c) There are 24 automorphisms with Lefschetz number $2$. They all have order $2$. 
d) There are 57 automorphisms with Lefschetz number $0$. They all are fixed point free. 
(With point we mean simplex). 

\paragraph{}
We also defined $L(G)$ as the {\bf average Lefschetz number} $\frac{1}{|{\rm Aut}(G)|} \sum_T \chi_T(G)$
and showed using Burnside's lemma 
that $L(G)$ can be interpreted as the Euler characteristic of an ``orbifold" $G/{\rm Aut}(G)$
which, under some non-degeneracy conditions, is again a graph. 
In the case $G=MK$, it is $1$. We also defined the {\bf Lefschetz curvature}
$$ \kappa(x)= \frac{1}{|{\rm Aut}(G)|} \sum_{T \in {\rm Aut}_x(G)} i_T(x) \; , $$
where ${\rm Aut}_x(G)$ is the stabilizer group of $x$. The Lefschetz fixed point theorem
implied the ``{\bf Gauss-Bonnet} type formula" $\sum_{x \in G} \kappa(x) = L(G)$. 
If $G$ is the complex of $MK$, the Lefschetz curvature is constant $\kappa(x)=1/16$ 
on each of the 16 vertices and zero on each of the edges. It adds up to $L(G)$.

\section{Natural groups}

\paragraph{}
The {\bf Pauli group} $P(1)$ is one of the smallest non-Abelian
groups relevant in quantum information theory. 
The Pauli matrices $X,Y,Z$ represent spin observables for spin 
Spin-$1/2$ particles. The Pauli group so bridges
quantum matrices, group theory and the quaternions, a division algebra.
Cayley graphs like $MK$ allows one to visualize the group structure of
$P(1)$. 

\paragraph{}
We have defined a group $(G,+)$ to be {\bf natural} if there exists a
metric space $X$ which carries a group structure
$(X,*)$ such that $(G,+)$ and $(X,*)$ are isomorphic groups
and so that the group multiplications and inversions are isometries of $X$.
The integers $\mathbb{Z}$ are not natural because it carries two groups,
the Abelian $\mathbb{Z}$ and the non-Abelian $D_{\infty}$, the infinite dihedral group 
$D_{\infty} = \mathbb{Z}  \rtimes C_2$. The latter is natural because we can 
build a double strung of $\mathbb{Z}$ which now admits only $D_{\infty}$ as
a group structure.

\begin{lemma}
The Pauli group $P(1)$ is natural: there is a metric space which admits only
one group operation $P(1)$ on its underlying set up to relabeling.
\end{lemma}
\begin{proof}
As $P(1)$ acts on $MK$, the group $P(1)$ is a subgroup of $Aut(MK)$. With the
usual graph distance metric there are other non-isomorphic groups acting on $MK$.
If we change the metric and render the edges corresponding to the different Pauli 
matrices generating $P(1)$ of different lengths 
we break the  $S_3$ symmetry that was present and gave $16*6$ automorphisms.
The symmetry group such a metric space $(MK,dist)$ 
must be $P(1)$ (up to relabeling). The group structure is forced if
we fix a vertex in $MK$ and call it $1$ (the neutral element in the group)
The 3 neighbors of $0$ have group elements X,Y,Z assigned which all are involutions.
As the group translations are isometries, the Pauli group structure is forced.
The relations are given by the geometry. There are cycles $(XYZ)^2$ or $(XY)^4$. 
\end{proof} 

{\bf Remark:} there is a relation with geometry:
for each of the 6 possible ordered pairs in $\{X,Y,Z\}$ like $X,Y$ 
define a path in $MK$ that can be seen as a geodesic $XYXYXY,...XY$ in the host
graph $G=K(2,2,2,2)$. 
Each of these geodesics which have length $8$ so that $(XY)^4=1$. 
The same holds for all the other pairs. Going along $XYZXYZ$ returns to 
the same point (these are the hexagonal faces in $MK$). 
There are no further relations.

\paragraph{}
Note that $P(1) = E \rtimes C(2)$, with a non-natural group
$E=C(4) \times C(2)$. We have shown that if
two groups $G,H$ are natural, then $G \rtimes H$ is natural by 
taking a Zig-Zag product on a geometric level.
But the reverse is not true. A group $G \rtimes H$ can be natural
even so $G$ is not natural. The basic example is that $C(4)$ is not natural
but that $D(8) = C(4) \rtimes C(2)$ is natural. 

\paragraph{}
Let us revisit the discrete cylinder graph $C_n \square C_2$, which is the 
{\bf prism} over the cyclic graph $C_n$ as well as the {\bf Moebius ladder} (also
with square faces) $M_n$. Under the assumption that $n \geq 4$, 
the automorphism group of $M_n$ is $D(2n)$ while the automorphism group 
of the untwisted prism is $D(n) \times C(2)$. Both the prism as well as the
moebius ladder are natural metric spaces, their metric forces a unique group structures.
The groups $D(2n)$ and $D(n) \times C(2)$ are both natural. 

\section{Remarks}

\paragraph{}
In this expository note, we have used a specific graph $MK$ to illustrate a few topics
that have interested us over time. They range from Cayley graph topics 
appearing in games like the Rubik's Cube, to unusual properties of the hypercube.
Barycentric refinements produce a connection with Julia sets of quadratic maps. 
The relations with graph coloring for manifolds parallel the classical Heawood theory of
surface coloring. The interplay of metric spaces and groups is fascinating, because
some metric spaces can force a group structure. We have seen this for the Pauli group.
We also took the opportunity to revisit the Lefschetz fixed point theorem in the
case of a graph $MK$ that has many automorphisms. Finally, the example can
illustrate intersection calculus.  
The main reason why to elaborate on this example is because we still struggle
to get a good notion of sectional curvature in discrete manifolds. The topic here has
emerged when playing with geodesic flow and parallel transport to construct geodesic sheets. 

\paragraph{}
It was a bit surprising to see the graph $MK$ is related to so many different popular topics.
As a collector of media containing mathematics and as someone who is
interested in visualizing mathematics (see \cite{CFZ}), the topic discussed here is ideal. 
There is an obvious danger that pop-culture topics can lead to overexposure. 
But it should not mean that one should stop thinking about it. 
A mathematical object enters into {\bf popular culture}, if it surfaces in movies, TV shows or
social media content, if it is realized in art, or if it appears
abundantly in popular mathematics books.
The tesseract for example has a cameo in a painting by the Spanish 
painter Dal\'i, it appears in movies
like {\bf hypercube}. Hamiltonian cycles have motivated the icosa game by Hamilton, Cayley groups
appear as puzzles, like the Rubik's Cube. A genus-2 Cayley graph of the Tucker group has been made
as a sculpture by De Witt Godfrey and Duane Martinez as discussed in \cite{Ferguson}. 

\begin{figure}[!htpb]
\scalebox{0.05}{\includegraphics{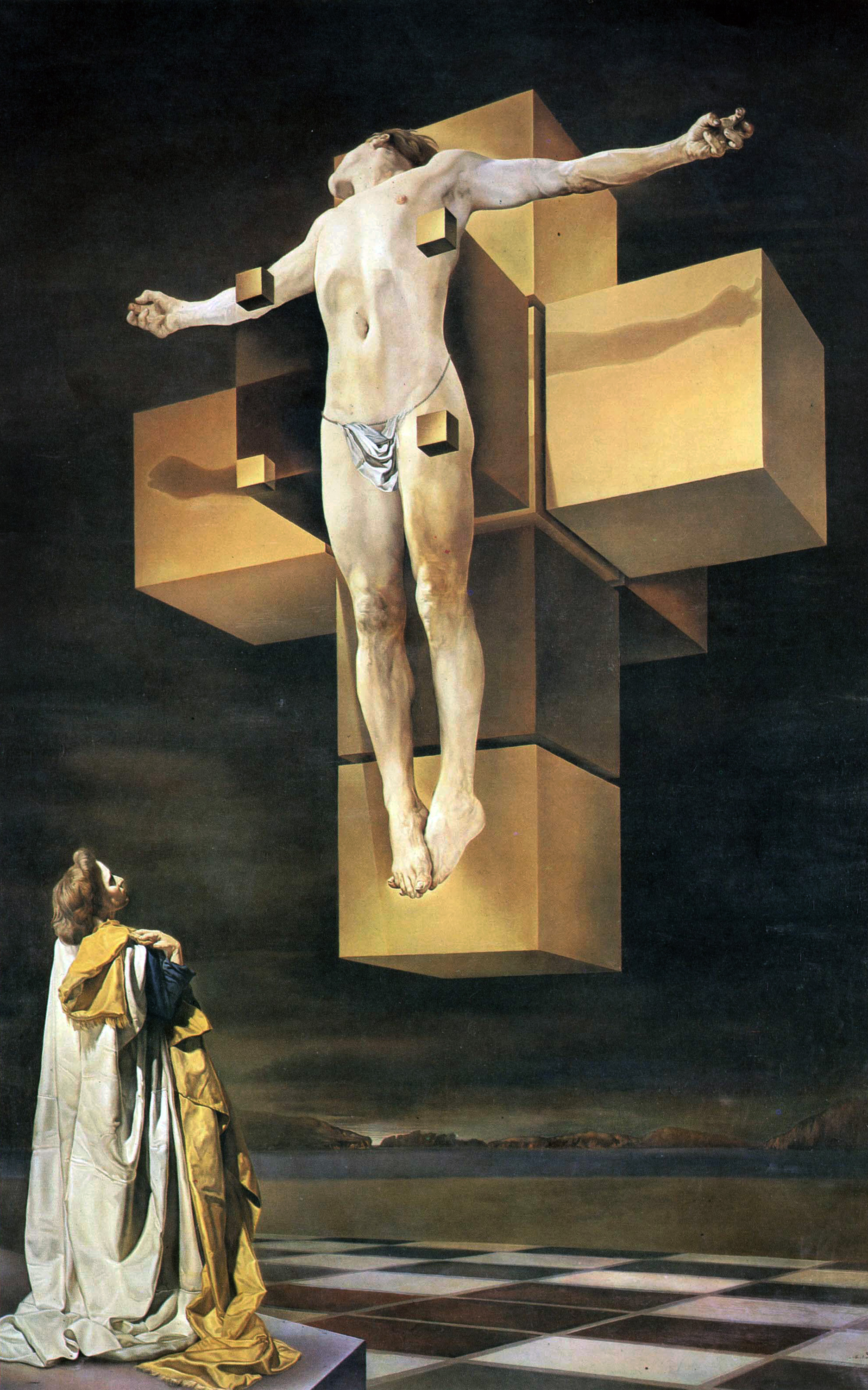}}
\scalebox{0.137}{\includegraphics{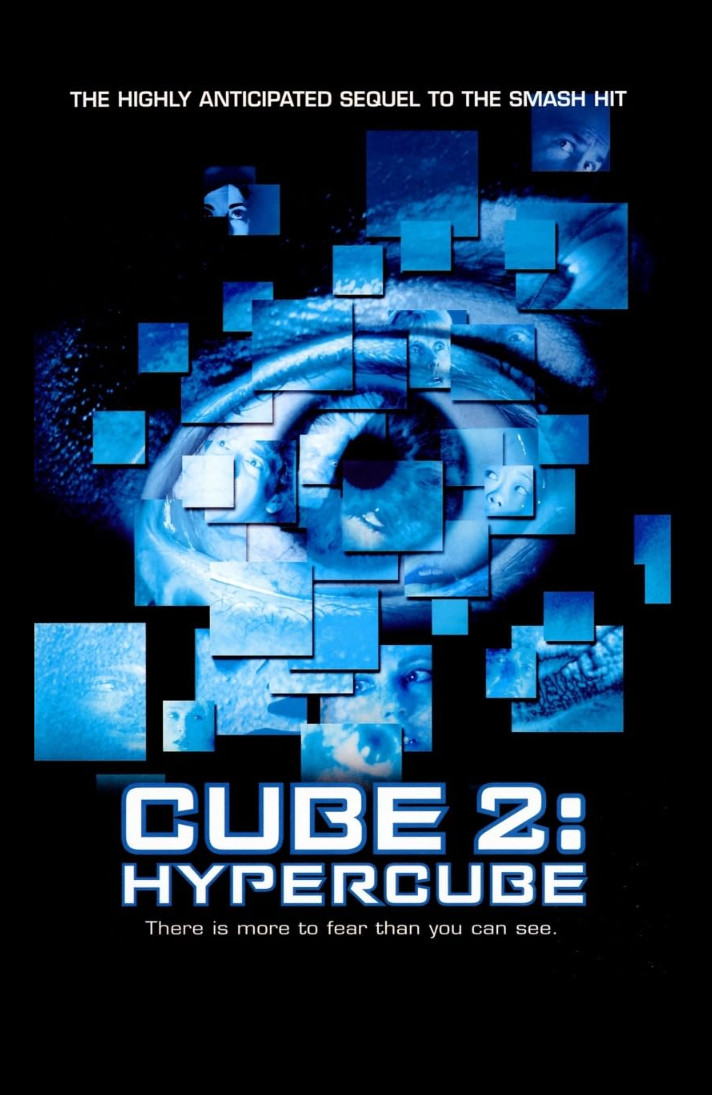}}
\scalebox{0.123}{\includegraphics{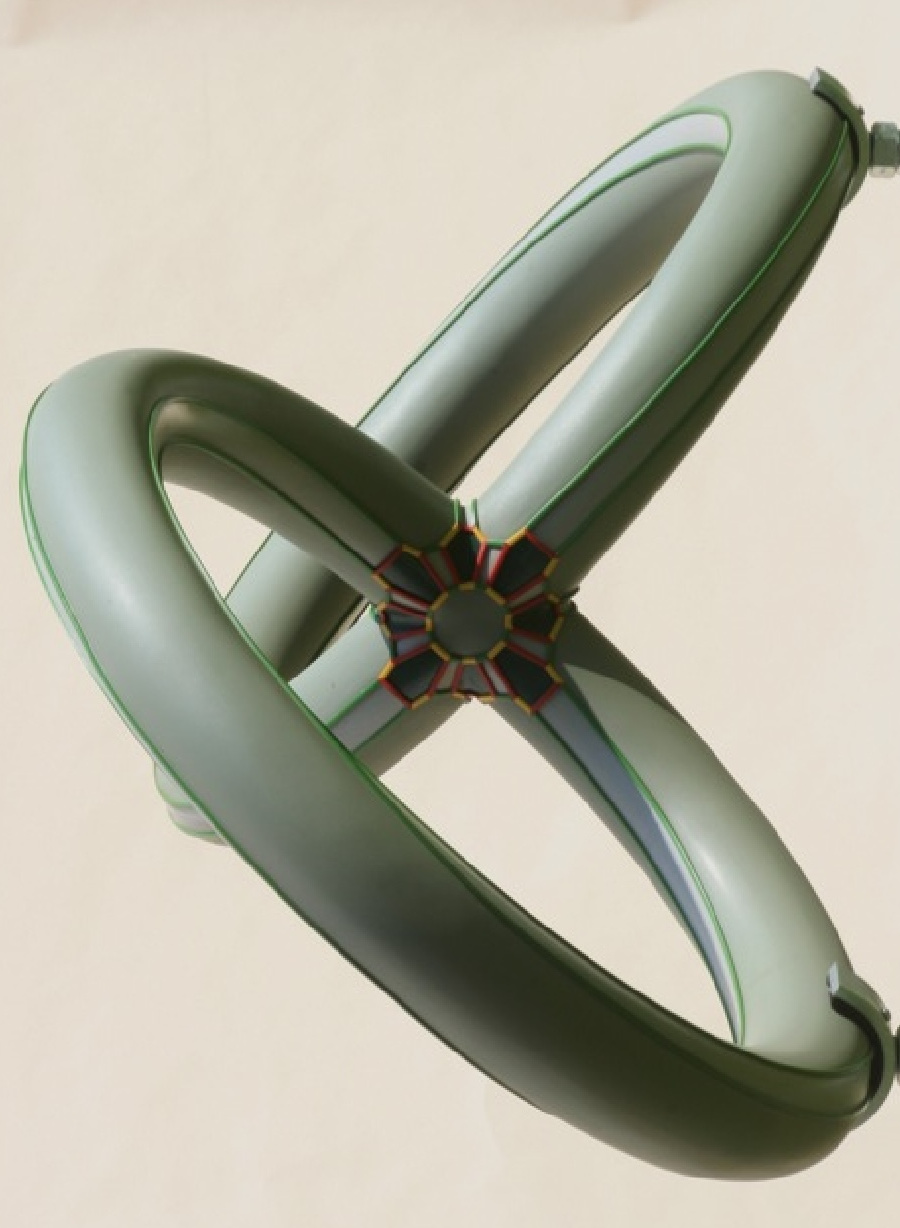}}
\label{popculture}
\caption{
Hypercubes have long been in pop-culture: the Hypercubus by Salvador Dal\'i from 1954.
The middle figure shows the DVD cover of the Hypercube~2 movie that appeared in 2002.
To the right, we see a sculpture by De Witt Godfrey and Duane Martinez showing a Cayley
graph of the Tucker group.
}
\end{figure}

\paragraph{}
Mathematical field like probability theory or game theory have emerged from
playing around. The nature of games very much depends on the number of players. 
One player games are solitaires or puzzles. Also in two player games, cooperation still does not matter, 
the set-up is simple and can be done purely graph theoretically \cite{Games2026}.
The Rubik's Cube is played still today, 50 years after its first emergence.
\footnote{There are about 150'000 speed-cubers participating in WCA events as of 2026. }

\paragraph{}
The 4-dimensional polytopes have all been animated many times.
Platonic solids have escaped pure mathematics and entered the realm of 
sacred geometry \cite{Skinner}, into symbolism like relations with
fire, earth, air, water, and ether which go back to 
Plato's Timaeus \cite{PlatoTimaeus}. Later there were
speculations about connections with celestial objects,
as in Kepler's {\it Harmonices Mundi}.

\begin{figure}[!htpb]
\scalebox{0.723}{\includegraphics{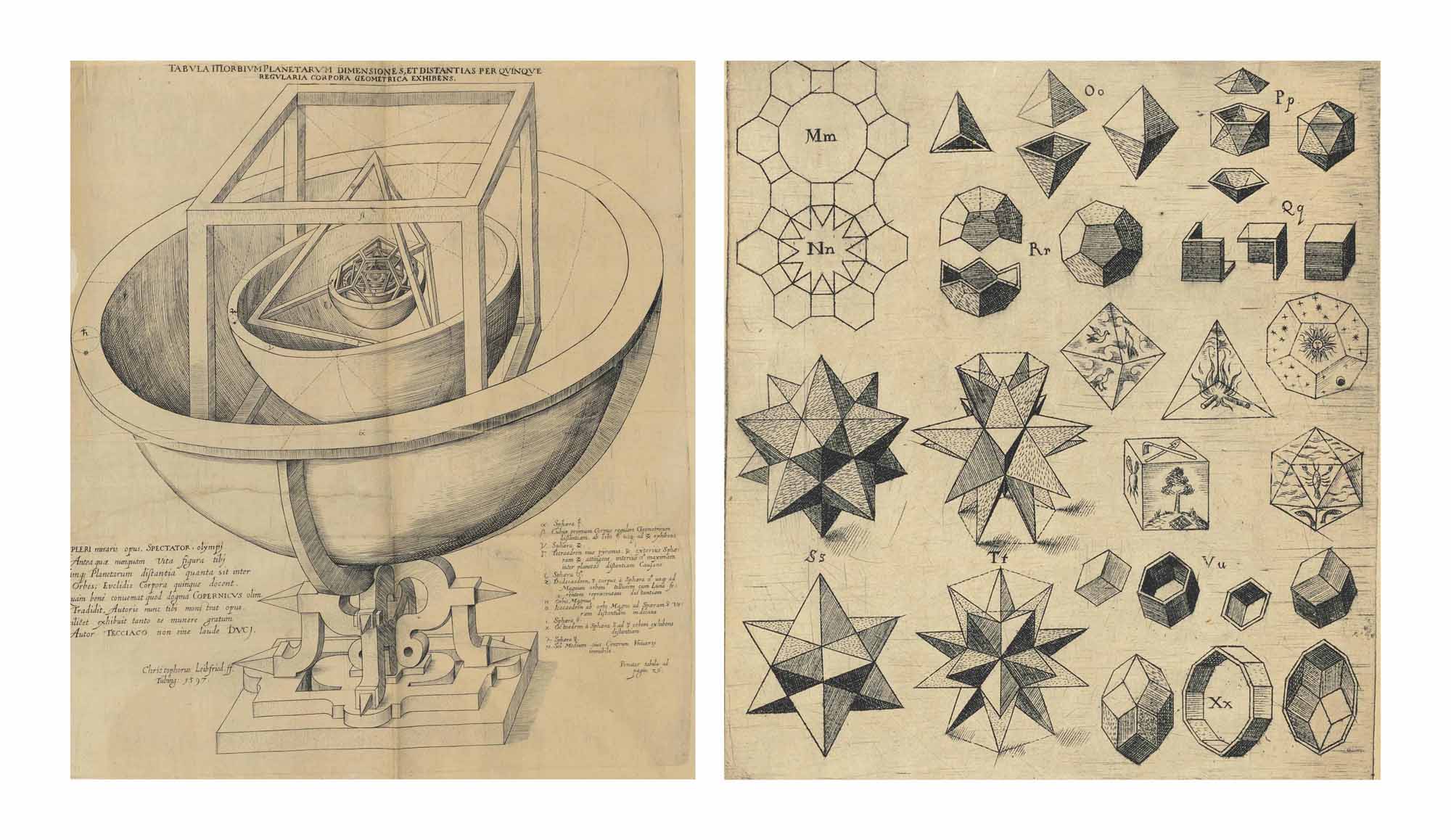}}
\label{harmonices}
\caption{
Harmonices Mundi by Kepler.
}
\end{figure}

\paragraph{}
The Pauli matrices are pivotal in quantum mechanics. 
Pauli's phrase ``not even wrong" has become a motif. 
The connection of Rubik's Cube with quarks appears
in \cite{Golomb}.  The 3-sphere was visualized using Hopf
fibrations in Banchoff's book \cite{Banchoff1990}.
Classically, the Clifford tori are $\{ |z|=\cos(t),|w|=\sin(t) \}$ in 
$S^3=\{(z,w) \in \mathbb{C}^2,  |z|^2+|w|^2=1 \}$,
where $0 < t < \pi/2$. In the limiting case $t=0$ or $t=\pi/2$, the tori
degenerate to circles. 

\begin{figure}[!htpb]
\scalebox{0.1}{\includegraphics{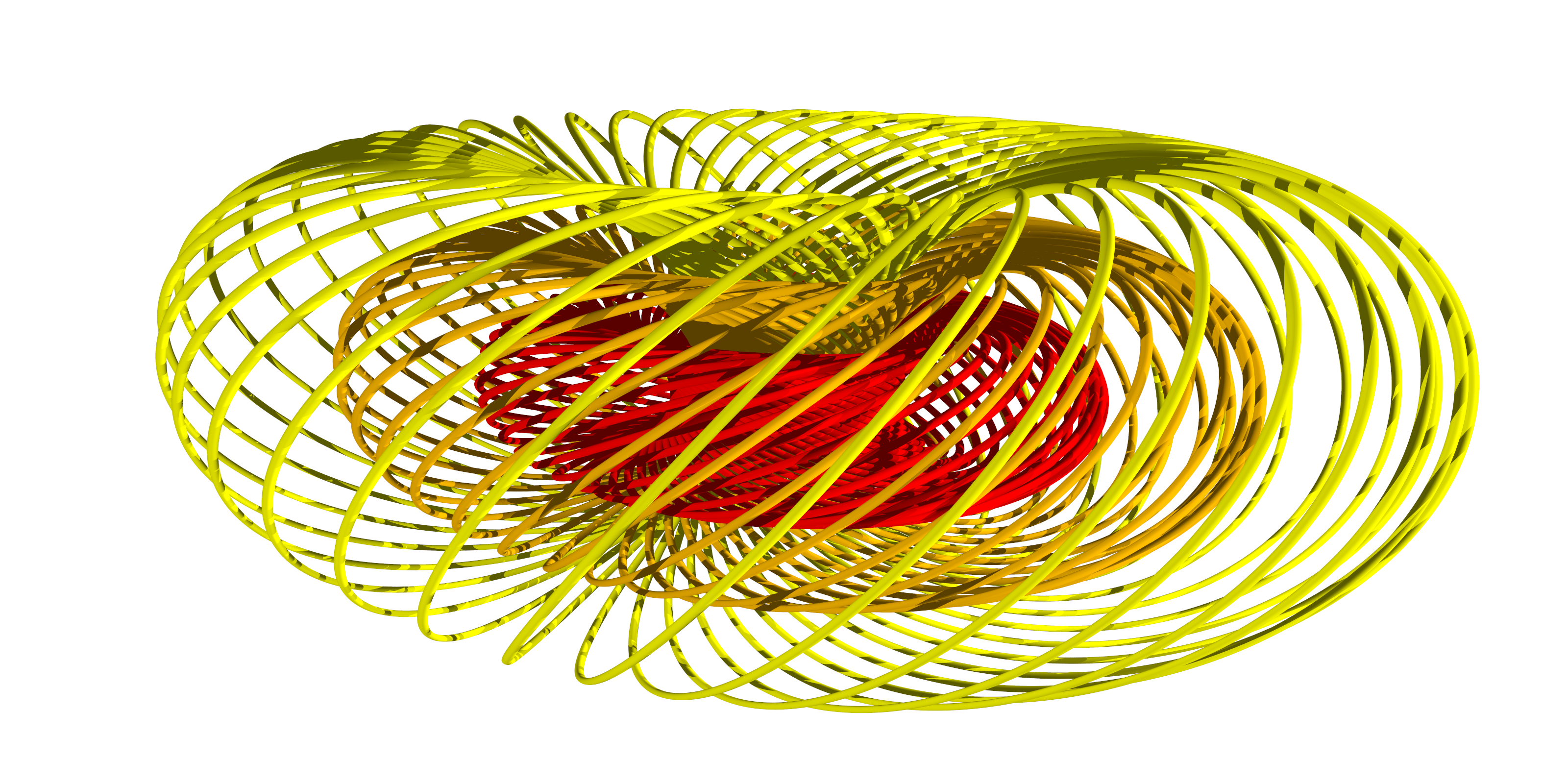}} 
\label{hopf}
\caption{
Clifford tori to visualize the 3-sphere.
}
\end{figure}

\paragraph{}
Aesthetic concepts also matter in mathematics \cite{AestheticMeasure}. 
The Wolfram language displays with
``StyleGraphs" structures embeddings \cite{WeisssteinMoebiusKantor}. Here are
20 displays from that list, selected and ordered according to our personal 
preferences. My own criteria for the ordering were
symmetry, surprise, and equi-distribution (no clutter), which of course are 
personal. Birkhoff tried to quantify beauty as ``Order/Complexity", O/C.
Complexity is for Birkhoff measured by the amount of effort to 
apprehend an object, while order means symmetry or harmony. 

\begin{figure}[!htpb]
\scalebox{0.8}{\includegraphics{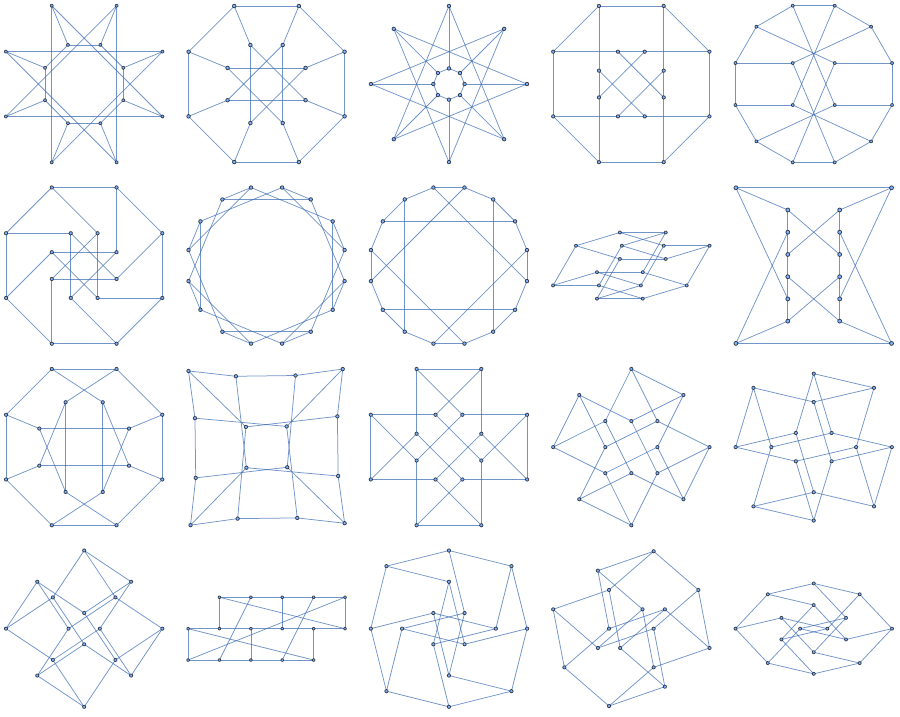}} 
\label{hopf}
\caption{
Embeddings of MK from the Wolfram library. After making our own choice , we
fed this picture to GPT-5.5 and asked to quantify beauty in the spirit
of Birkhoff. Here is what the machine answered: 
{\it a possible Birkhoff-like scoring scheme 
for these graph drawings would be: O = S + R + P + B
where S counts symmetries, R counts repeated motifs, 
P measures parallelism or equal spacing, and B measures global balance. Then
C = V + E + X + I, where V is the number of vertices, E the 
number of edges, X the number of crossings, and I an irregularity penalty. 
A decisive factor would be O, plus how well the crossings are organized.}
}
\end{figure}

\paragraph{}
MK produces a map on 
the 2-torus which needs 7 colors. Cheaper is to embed the complete graph $K_7$ into 
the 2-torus. We feel that MK evident using MK. 
An exciting open problem in the field of {\bf manifold coloring} is to
answer the {\bf Albertson and Stromquist question}, whether every torus 
can be 5 colored. This more topological graph coloring problem has 
A rather general upper bound for $q$-manifold 
is $2q+2$, but that does not exclude that there could be tori 
of chromatic number $6$. 

\paragraph{}
The argument for the upper chromatic bound $2q+2$ is with respect to dimension $q$ 
and uses that the dual of
every q-manifold has the {\bf vertex arboricity} $2$. Now color facets belonging to one sort of tree
with $q+1$ colors and the others with an other set of $q+1$ colors. A q=1 manifold
(a union of circular graph of length $4$ or larger) is self dual and has vertex arboricity
2. To color a 2-torus with 6 colors, cut it into two cylinders, look at the dual graph, 
color the boundary with 2 colors and grow trees into the interior, coloring each triangle
so in two different colors. Each of the triangle colors now takes a set of 3 vertex
colors. The Albertson-Stromquest question is now: {\bf ``are 5 colors enough on a torus?"}. 
It is a completely different problem than the Heawood question from topological
graph theory. And MK illustrates, $MK^*$ is what topological graph theorists would
call a ``discrete torus", a triangulation of the torus. But is not manifold like from a 
finite geometry perspective as even the smallest ball in that graph is a wheel graph,
where the boundary is a figure 8 graph. It has not the local manifold property in the
terminology of Albertson-Stromquest, or is not a discrete 2-manifold in the sense that
all unit spheres are circular graphs with 4 or more elements. 

\paragraph{}
Let us use the example to illustrate some spectral graph theory \cite{Knill2024}. We proved the 
inequality $\lambda_k \leq d_k + d_{k-1}$ for the eigenvalues $\lambda_k$ of any graph or
quiver, where $d_k$ is the $k$'th ordered vertex degrees, assuming $d_0=0$ 
(both sequences are arranged increasingly). For a $3$-regular graph, 
this only gives $\lambda_k \leq 6$. The Kirchhoff
matrix of $MK$ has eight integer eigenvalues $0,2,2,2,4,4,4,6$ and two
algebraic eigenvalues $3 \pm \sqrt{3}$, each of multiplicity $4$. 
The Hodge Laplacian $L_0 \oplus L_1$ (which is a $40 \times 40$ matrix) has
twice as many nonzero eigenvalues by the McKean-Singer symmetry \cite{McKeanSinger,knillmckeansinger}
and additionally $9$ more eigenvalues $0$ as the Betti numbers are 
$b_0=1$ (it is a connected graph) and $b_1=|E|-|V|+1=24-16+1=9$ 
(by the Euler-Poincar\'e formula $\chi(MK)=|V|-|E|=b_0-b_1=-8$). 
If the 8 hexagonal loops in MK were counted as faces, 
we would obtain a CW complex $M$ that has Euler characteristic 
$\chi(M) = |V|-|E|+|F|=16-24+8=0$. 

\paragraph{}
The graph $MK$ is one-dimensional. The Barycentric refinements agree with 
those usually used in topological graph theory.
\footnote{As there are no triangles, we only need to consider 
          zero- and one-dimensional simplices}. 
Barycentric refinements help to visualize graphs. The spectrum of the Laplacian converges
in law to the universal law in one-dimensions, the {\bf arcsine-distribution}
on $[0,4]$ with cumulative distribution function $(2/\pi) \arcsin(\sqrt{x}/2)$.
Spectral universality occurs in any dimension $q$: there is a universal 
limiting density of states which only depends on $q$ \cite{KnillBarycentric2} and
is is non-trivial and not explored in dimension $d>1$. 
For $q=1$, the limiting spectral function is known and related to the {\bf Julia set} of the
{\bf quadratic map} $z \mapsto z(4-z)$, conjugated to
$z \mapsto z^2-2$ ($c=-2$ at lowest tip of the Mandelbrot set) or to the {\bf Ulam map}
$z \mapsto 4z(1-z)$.

\begin{figure}[!htpb] 
\scalebox{0.5}{\includegraphics{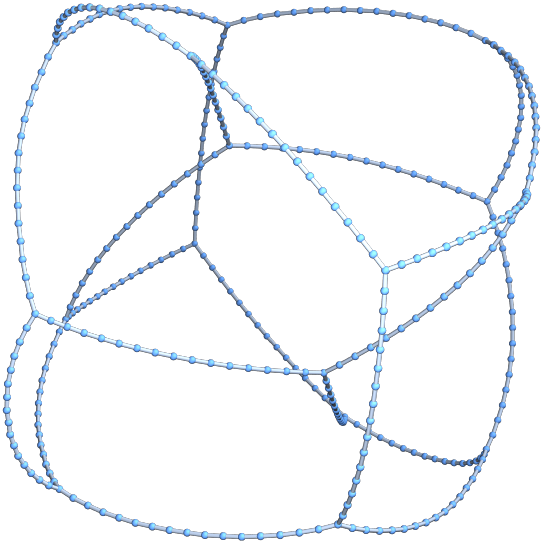}}
\scalebox{0.5}{\includegraphics{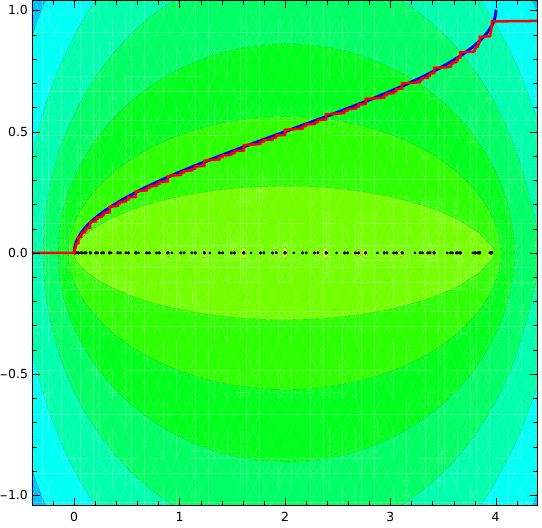}}
\label{Barycentric}  
\caption{
The fourth Barycentric refinement $MK_4$ of MK. The spectrum 
of the Kirchhoff Laplacians of successive refinements $MK_n$ converge in law
to the arcsine distribution, the unique equilibrium measure on the 
a Julia set for the quadratic map. 
}
\end{figure}

\paragraph{}
In {\bf connection calculus}, one can look at the (here one-dimensional)
simplicial complex $G=V \cup E$ of the graph and define the 
{\bf connection matrix} $L$ as $L(x,y)=\chi(C(x) \cap C(y))$,
where $C(x) = \{ y \in G, y \subset x\}$ and where
$\chi(A) = \sum_{y \in A} \omega(y)$ is the {\bf Euler characteristic}
with $\omega(x)=(-1)^{{\rm dim}(x)}$ so that $\chi(G) = |V|-|E|=-8$. 
\footnote{We always mean with $\subset$ the subset operation, not the strict subset operation
$x \subset x$}.
The unimodular matrix $L$ has as its inverse the integer {\bf Green function matrix}
$g(x,y) = \omega(x) \omega(y) \chi(U(x) \cap U(y))$, where 
$U(x)=\{ y \in G, x \subset y\}$ is the {\bf star} of $x \in G$. The statement $g=L^{-1}$ was 
dubbed {\bf Green star formula} because $g$ is a Green's function and stars are involved in the
expressions for the entry. 
The determinant of $L$ equal to the {\bf Fermi characteristic} 
$\prod_{x \in G} \omega(x)$, which is
a multiplicative version of the Euler characteristic $\chi(G) = \sum_{x \in G} \omega(x)$.
We have also shown that the number of positive eigenvalues of $L$ minus the number of negative
eigenvalues of $L$ is the Euler characteristic of $G$ \cite{KnillEnergy2020}.
In the case of $MK$, we have $16$ positive eigenvalues of $L$ and $24$ negative 
eigenvalues of the connection Laplacian $L$. 
To investigate whether one can ``hear the Euler characteristic" of a simplicial complex  
\cite{HearingEulerCharacteristic}
was inspired by the inverse spectral question \cite{Kac66}. As $MK$ is one-dimensional, one
can relate the Hodge (intersection Laplacian) $H$ and connection Laplacian $L$ via the
{\bf Hydrogen identity} $L-L^{-1}=|H|$, where
$H$ is the sign-less Hodge Laplacian. The spectral zeta function of the connection
Laplacian satisfies the functional equation 
$\zeta_{L^2}(-s) = \zeta_{L^2}(s)$ \cite{DyadicRiemann}. 

\begin{figure}[!htpb] 
\scalebox{1.0}{\includegraphics{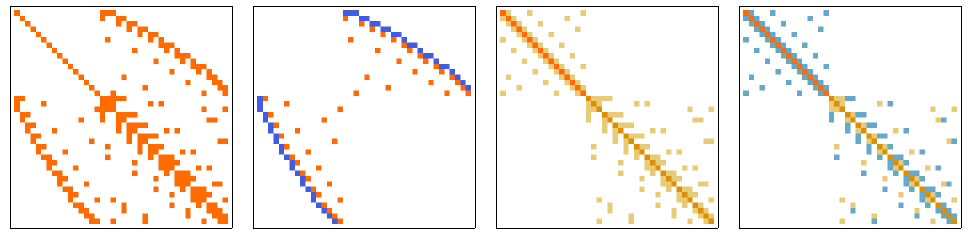}}
\label{Hydrogen}  
\caption{
We see the Dirac matrix $D=d+d^*$ and the connection matrix $L$ of MK. 
The third matrix is the Hydrogen operator $L-L^{-1}$ which is related to 
the Hodge Laplacian $H=D^2$ at the end by $L-L^{-1} = |D|^2$ The Hodge and
Hydrogen matrices are block diagonal. The first block is the Kirchhoff matrix
rsp. the sign-less Kirchhoff matrix. 
}
\end{figure}

\bibliographystyle{plain}

\end{document}